\documentclass{amsart}
\usepackage{amsfonts}

\newtheorem{theorem}{Theorem}
\theoremstyle{plain}

\newtheorem{corollary}{Corollary}

\newtheorem{definition}{Definition}

\newtheorem{lemma}{Lemma}

\newtheorem{proposition}{Proposition}
\newtheorem{remark}{Remark}

\numberwithin{equation}{section}
\input{tcilatex}

\begin{document}
\title[Some Hadamard like inequalities]{Some Hadamard like inequalities via
convex and $s$-convex functions and their applications for special means}
\author{Mevl\"{u}t TUN\c{C}}
\address{Kilis 7 Aral\i k University, Department of Mathematics, 79000,
Kilis, TURKEY}
\email{mevluttunc@kilis.edu.tr}
\urladdr{}
\date{}
\thanks{This paper is in final form and no version of it will be submitted
for publication elsewhere.}

\begin{abstract}
In this study, the author establish some inequalities of Hadamard like based
on convex and $s$-convexity in the second sense. Some applications to
special means of positive real numbers are also given.
\end{abstract}

\maketitle

\section{Preliminaries}

\subsection{Definitions}

\begin{definition}
\cite{mit2} A function $f:I\rightarrow 
\mathbb{R}
$ is said to be convex on $I$ if inequality%
\begin{equation}
f\left( tx+\left( 1-t\right) y\right) \leq tf\left( x\right) +\left(
1-t\right) f\left( y\right)  \label{101}
\end{equation}%
holds for all $x,y\in I$ and $t\in \left[ 0,1\right] $. We say that $f$ is
concave if $(-f)$ is convex.
\end{definition}

Geometrically, this means that if $P,Q$ and $R$ are three distinct points on
the graph of $f$ with $Q$ between $P$ and $R$, then $Q$ is on or below the
chord $PR$.

\begin{definition}
\cite{hud}\textit{\ Let }$s\in \left( 0,1\right] .$\textit{\ A function }$%
f:\left( 0,\infty \right] \rightarrow \left[ 0,\infty \right] $\textit{\ is
said to be }$s$-\textit{convex in the second sense if \ \ \ \ \ \ \ \ \ \ \
\ }%
\begin{equation}
f\left( tx+\left( 1-t\right) y\right) \leq t^{s}f\left( x\right) +\left(
1-t\right) ^{s}f\left( y\right) ,  \label{105}
\end{equation}%
\textit{for all }$x,y\in \left( 0,b\right] $\textit{\ \ and }$t\in \left[ 0,1%
\right] $\textit{. This class of }$s$\textit{-convex functions is usually
denoted by }$K_{s}^{2}$\textit{.}
\end{definition}

Certainly, $s$-convexity means just ordinary convexity when $s=1$.

\subsection{Theorems}

\begin{theorem}
\textbf{The Hermite-Hadamard inequality:} Let $f:I\subseteq 
\mathbb{R}
\rightarrow 
\mathbb{R}
$ be a convex function and $u,v\in I$ with $u<v$. The following double
inequality:%
\begin{equation}
f\left( \frac{u+v}{2}\right) \leq \frac{1}{v-u}\int_{u}^{v}f\left( x\right)
dx\leq \frac{f\left( u\right) +f\left( v\right) }{2}  \label{110}
\end{equation}%
is known in the literature as Hadamard's inequality (or Hermite-Hadamard
inequality) for convex functions. If $f$ is a positive concave function,
then the inequality is reversed.
\end{theorem}

\begin{theorem}
\cite{ssd6} \textit{Suppose that }$f:\left[ 0,\infty \right) \rightarrow %
\left[ 0,\infty \right) $\textit{\ is an }$s-$\textit{convex function in the
second sense, where }$s\in \left( 0,1\right] $\textit{, and let }$a,b\in %
\left[ 0,\infty \right) ,$ $a<b.$\textit{\ If }$f\in L_{1}\left( \left[ 0,1%
\right] \right) $\textit{, then the following inequalities hold:}%
\begin{equation}
2^{s-1}f\left( \frac{u+v}{2}\right) \leq \frac{1}{v-u}\int_{u}^{v}f\left(
x\right) dx\leq \frac{f\left( u\right) +f\left( v\right) }{s+1}.  \label{109}
\end{equation}%
The constant $k=\frac{1}{s+1}$ is the best possible in the second inequality
in (\ref{109}). The above inequalities are sharp. If $f$ is an $s$-concave
function in the second sense, then the inequality is reversed.
\end{theorem}

For recent results and generalizations concerning Hadamard's inequality and
concepts of convexity and $s$-convexity see \cite{alo2}-\cite{yang} and the
references therein.

Throughout this paper we will use the following notations and conventions.
Let $J=\left[ 0,\infty \right) \subset 
\mathbb{R}
=\left( -\infty ,+\infty \right) ,$ and $u,v\in J$ with $0<u<v$ and $%
f^{\prime }\in L\left[ u,v\right] $ and 
\begin{eqnarray*}
A\left( u,v\right) &=&\frac{u+v}{2},\text{ }G\left( u,v\right) =\sqrt{uv},%
\text{ }I\left( u,v\right) =\frac{1}{e}\left( \frac{b^{b}}{a^{a}}\right) ^{%
\frac{1}{b-a}}, \\
L_{p}\left( u,v\right) &=&\left( \frac{v^{p+1}-u^{p+1}}{\left( p+1\right)
\left( v-u\right) }\right) ^{1/p},\text{ }u\neq v,\text{ }p\in 
\mathbb{R}
,\text{ }p\neq -1,0
\end{eqnarray*}%
be the arithmetic mean, geometric mean, identric mean, generalized
logarithmic mean for $u,v>0$ respectively.

\section{Some new Hadamard like inequalities}

In order to establish our main results, we first establish the following
lemma.

\begin{lemma}
\label{ll}Let $f:J\rightarrow 
\mathbb{R}
$ be a differentiable function on $J^{\circ }$. If $f^{\prime }\in L\left[
u,v\right] ,$ then%
\begin{eqnarray*}
&&\frac{\left( v-x\right) \left( vf\left( v\right) -uf\left( x\right)
\right) +\left( x-u\right) \left( vf\left( x\right) -uf\left( u\right)
\right) }{\left( v-u\right) ^{2}}-\frac{1}{v-u}\int_{u}^{v}f\left( \mu
\right) d\mu \\
&=&\frac{\left( v-x\right) ^{2}}{\left( v-u\right) ^{2}}\int_{0}^{1}\left(
tu+\left( 1-t\right) v\right) f^{\prime }\left( tx+\left( 1-t\right)
v\right) dt \\
&&+\frac{\left( x-u\right) ^{2}}{\left( v-u\right) ^{2}}\int_{0}^{1}\left(
tv+\left( 1-t\right) u\right) f^{\prime }\left( tx+\left( 1-t\right)
u\right) dt
\end{eqnarray*}%
for each $t\in \left[ 0,1\right] $ and $x\in \left[ u,v\right] .$
\end{lemma}

\begin{proof}
Integrating by parts, we get%
\begin{eqnarray*}
&&\frac{\left( v-x\right) ^{2}}{\left( v-u\right) ^{2}}\int_{0}^{1}\left(
tu+\left( 1-t\right) v\right) f^{\prime }\left( tx+\left( 1-t\right)
v\right) dt \\
&&+\frac{\left( x-u\right) ^{2}}{\left( v-u\right) ^{2}}\int_{0}^{1}\left(
tv+\left( 1-t\right) u\right) f^{\prime }\left( tx+\left( 1-t\right)
u\right) dt \\
&=&\frac{\left( v-x\right) ^{2}}{\left( v-u\right) ^{2}}\left[ \left. \left(
tu+\left( 1-t\right) v\right) \frac{f\left( tx+\left( 1-t\right) v\right) }{%
x-v}\right\vert _{0}^{1}-\left( u-v\right) \int_{0}^{1}\frac{f\left(
tx+\left( 1-t\right) v\right) }{x-v}dt\right] \\
&&+\frac{\left( x-u\right) ^{2}}{\left( v-u\right) ^{2}}\left[ \left. \left(
tv+\left( 1-t\right) u\right) \frac{f\left( tx+\left( 1-t\right) u\right) }{%
x-u}\right\vert _{0}^{1}-\left( v-u\right) \int_{0}^{1}\frac{f\left(
tx+\left( 1-t\right) u\right) }{x-u}dt\right] \\
&=&\frac{\left( v-x\right) ^{2}}{\left( v-u\right) ^{2}}\left[ \frac{%
uf\left( x\right) -vf\left( v\right) }{x-v}-\frac{v-u}{\left( x-v\right) ^{2}%
}\int_{x}^{v}f\left( \mu \right) d\mu \right] \\
&&+\frac{\left( x-a\right) ^{2}}{\left( v-u\right) ^{2}}\left[ \frac{%
vf\left( x\right) -uf\left( u\right) }{x-u}-\frac{v-u}{\left( x-u\right) ^{2}%
}\int_{u}^{x}f\left( \mu \right) d\mu \right] \\
&=&\frac{\left( v-x\right) \left( vf\left( v\right) -uf\left( x\right)
\right) +\left( x-u\right) \left( vf\left( x\right) -uf\left( u\right)
\right) }{\left( v-u\right) ^{2}}-\frac{1}{v-u}\int_{u}^{v}f\left( \mu
\right) d\mu .
\end{eqnarray*}
\end{proof}

\begin{theorem}
\label{t1}Let $f:J\rightarrow 
\mathbb{R}
$ be a differentiable function on $J^{\circ }.$ If $\left\vert f^{\prime
}\right\vert $\ is convex on $\left[ u,v\right] ,$ then%
\begin{eqnarray*}
&&\left\vert \frac{\left( v-x\right) \left( vf\left( v\right) -uf\left(
x\right) \right) +\left( x-u\right) \left( vf\left( x\right) -uf\left(
u\right) \right) }{\left( v-u\right) ^{2}}-\frac{1}{v-u}\int_{u}^{v}f\left(
\mu \right) d\mu \right\vert \\
&\leq &\frac{\left( v-x\right) ^{2}}{6\left( v-u\right) ^{2}}\left[ \left(
2u+v\right) \left\vert f^{\prime }\left( x\right) \right\vert +\left(
u+2v\right) \left\vert f^{\prime }\left( v\right) \right\vert \right] +\frac{%
\left( x-u\right) ^{2}}{6\left( v-u\right) ^{2}}\left[ \left( u+2v\right)
\left\vert f^{\prime }\left( x\right) \right\vert +\left( 2u+v\right)
\left\vert f^{\prime }\left( u\right) \right\vert \right]
\end{eqnarray*}%
for each $x\in \left[ u,v\right] .$
\end{theorem}

\begin{proof}
Using Lemma \ref{ll} and from properties of modulus, and since $\left\vert
f^{\prime }\right\vert $\ is convex on $\left[ u,v\right] ,$ then we obtain%
\begin{eqnarray*}
&&\left\vert \frac{\left( v-x\right) \left( vf\left( v\right) -uf\left(
x\right) \right) +\left( x-u\right) \left( vf\left( x\right) -uf\left(
u\right) \right) }{\left( v-u\right) ^{2}}-\frac{1}{v-u}\int_{u}^{v}f\left(
\mu \right) d\mu \right\vert \\
&\leq &\frac{\left( v-x\right) ^{2}}{\left( v-u\right) ^{2}}%
\int_{0}^{1}\left( tu+\left( 1-t\right) v\right) \left\vert f^{\prime
}\left( tx+\left( 1-t\right) v\right) \right\vert dt \\
&&+\frac{\left( x-u\right) ^{2}}{\left( v-u\right) ^{2}}\int_{0}^{1}\left(
tv+\left( 1-t\right) u\right) \left\vert f^{\prime }\left( tx+\left(
1-t\right) u\right) \right\vert dt \\
&\leq &\frac{\left( v-x\right) ^{2}}{\left( v-u\right) ^{2}}%
\int_{0}^{1}\left( tu+\left( 1-t\right) v\right) \left( t\left\vert
f^{\prime }\left( x\right) \right\vert +\left( 1-t\right) \left\vert
f^{\prime }\left( v\right) \right\vert \right) dt \\
&&+\frac{\left( x-u\right) ^{2}}{\left( v-u\right) ^{2}}\int_{0}^{1}\left(
tv+\left( 1-t\right) u\right) \left( t\left\vert f^{\prime }\left( x\right)
\right\vert +\left( 1-t\right) \left\vert f^{\prime }\left( u\right)
\right\vert \right) dt \\
&=&\frac{\left( v-x\right) ^{2}}{6\left( v-u\right) ^{2}}\left[ \left(
2u+v\right) \left\vert f^{\prime }\left( x\right) \right\vert +\left(
u+2v\right) \left\vert f^{\prime }\left( v\right) \right\vert \right] \\
&&+\frac{\left( x-u\right) ^{2}}{6\left( v-u\right) ^{2}}\left[ \left(
u+2v\right) \left\vert f^{\prime }\left( x\right) \right\vert +\left(
2u+v\right) \left\vert f^{\prime }\left( u\right) \right\vert \right] .
\end{eqnarray*}%
The proof is completed.
\end{proof}

\begin{theorem}
\label{t2} Let $f:J\rightarrow 
\mathbb{R}
$ be a differentiable function on $J^{\circ }.$ If $\left\vert f^{\prime
}\right\vert $\ is $s$-convex on $\left[ u,v\right] $ for some fixed $s\in
\left( 0,1\right] ,$ then%
\begin{eqnarray*}
&&\left\vert \frac{\left( v-x\right) \left( vf\left( v\right) -uf\left(
x\right) \right) +\left( x-u\right) \left( vf\left( x\right) -uf\left(
u\right) \right) }{\left( v-u\right) ^{2}}-\frac{1}{v-u}\int_{u}^{v}f\left(
\mu \right) d\mu \right\vert \\
&\leq &\frac{\left( v-x\right) ^{2}}{\left( v-u\right) ^{2}}\left[ \frac{%
\left( \left( s+1\right) u+v\right) \left\vert f^{\prime }\left( x\right)
\right\vert +\left( u+\left( s+1\right) v\right) \left\vert f^{\prime
}\left( v\right) \right\vert }{\left( s+1\right) \left( s+2\right) }\right]
\\
&&+\frac{\left( x-u\right) ^{2}}{\left( v-u\right) ^{2}}\left[ \frac{\left(
u+\left( s+1\right) v\right) \left\vert f^{\prime }\left( x\right)
\right\vert +\left( \left( s+1\right) v+v\right) \left\vert f^{\prime
}\left( u\right) \right\vert }{\left( s+1\right) \left( s+2\right) }\right]
\end{eqnarray*}%
for each $x\in \left[ u,v\right] .$
\end{theorem}

\begin{proof}
The proof of Theorem \ref{t2} is similar to Theorem \ref{t1}.
\end{proof}

\begin{remark}
In Theorem \ref{t1}, if we take $s=1$, then Theorem \ref{t2} reduces to
Theorem \ref{t1}.
\end{remark}

\begin{corollary}
\label{c1} In Theorem \ref{t2}, if we choose $x=\frac{u+v}{2},$ we get%
\begin{eqnarray*}
&&\left\vert \frac{vf\left( v\right) -uf\left( u\right) }{2\left( v-u\right) 
}+\frac{1}{2}f\left( \frac{u+v}{2}\right) -\frac{1}{v-u}\int_{u}^{v}f\left(
\mu \right) d\mu \right\vert \\
&\leq &\frac{1}{4}\left[ \frac{\left( \left( s+1\right) u+v\right)
\left\vert f^{\prime }\left( \frac{u+v}{2}\right) \right\vert +\left(
u+\left( s+1\right) v\right) \left\vert f^{\prime }\left( v\right)
\right\vert }{\left( s+1\right) \left( s+2\right) }\right] \\
&&+\frac{1}{4}\left[ \frac{\left( u+\left( s+1\right) v\right) \left\vert
f^{\prime }\left( \frac{u+v}{2}\right) \right\vert +\left( \left( s+1\right)
u+v\right) \left\vert f^{\prime }\left( u\right) \right\vert }{\left(
s+1\right) \left( s+2\right) }\right]
\end{eqnarray*}
\end{corollary}

\begin{proposition}
Let $u,v\in J^{\circ },\ 0<u<v$ and $s\in \left( 0,1\right] ,$ then%
\begin{eqnarray}
&&\left\vert \frac{s+1}{2}L_{s}^{s}\left( u,v\right) +\frac{1}{2}A^{s}\left(
u,v\right) -L_{s}^{s}\left( u,v\right) \right\vert  \notag \\
&=&\left\vert \left( s-1\right) L_{s}^{s}\left( u,v\right) +A^{s}\left(
u,v\right) \right\vert  \label{q} \\
&\leq &\frac{s}{s+1}A\left( u,v\right) +\frac{s}{s+2}A\left(
u^{s},v^{s}\right) +\frac{sG^{2}\left( u,v\right) }{\left( s+1\right) \left(
s+2\right) }A\left( u^{s-2},v^{s-2}\right)  \notag
\end{eqnarray}
\end{proposition}

\begin{proof}
The proof follows from Corollary \ref{c1} applied\ to the $s$-convex
function $f\left( x\right) =x^{s}.$ Equality in (\ref{q}) holds if and only
if $s=1.$
\end{proof}

\begin{corollary}
\label{c2} In Corollary \ref{c1}, if we take $\left\vert f^{\prime
}\right\vert \leq M,$ we get%
\begin{equation*}
\left\vert \frac{vf\left( v\right) -uf\left( u\right) }{2\left( v-u\right) }+%
\frac{1}{2}f\left( \frac{u+v}{2}\right) -\frac{1}{v-u}\int_{u}^{v}f\left(
\mu \right) d\mu \right\vert \leq M\frac{\left( u+v\right) }{2\left(
s+1\right) }
\end{equation*}
\end{corollary}

\begin{theorem}
\label{t3}Let $f:J\rightarrow 
\mathbb{R}
$ be a differentiable function on $J^{\circ }.$ If $\left\vert f^{\prime
}\right\vert ^{q}$\ is convex on $\left[ u,v\right] $ and $q>1$ with $%
1/p+1/q=1$, then%
\begin{eqnarray*}
&&\left\vert \frac{\left( v-x\right) \left( vf\left( v\right) -uf\left(
x\right) \right) +\left( x-u\right) \left( vf\left( x\right) -uf\left(
u\right) \right) }{\left( v-u\right) ^{2}}-\frac{1}{v-u}\int_{u}^{v}f\left(
\mu \right) d\mu \right\vert \\
&\leq &\frac{\left( v-x\right) ^{2}}{\left( v-u\right) ^{2}}L_{p}\left(
u,v\right) \left( \frac{\left\vert f^{\prime }\left( x\right) \right\vert
^{q}+\left\vert f^{\prime }\left( v\right) \right\vert ^{q}}{2}\right) ^{%
\frac{1}{q}}+\frac{\left( x-u\right) ^{2}}{\left( v-u\right) ^{2}}%
L_{p}\left( u,v\right) \left( \frac{\left\vert f^{\prime }\left( x\right)
\right\vert ^{q}+\left\vert f^{\prime }\left( u\right) \right\vert ^{q}}{2}%
\right) ^{\frac{1}{q}}.
\end{eqnarray*}%
for each $x\in \left[ u,v\right] .$
\end{theorem}

\begin{proof}
Using Lemma \ref{ll} and using the well-known H\"{o}lder's inequality and
since $\left\vert f^{\prime }\right\vert ^{q}$\ is convex on $\left[ u,v%
\right] ,$ we establish%
\begin{eqnarray}
&&  \label{2} \\
&&\left\vert \frac{\left( v-x\right) \left( vf\left( v\right) -uf\left(
x\right) \right) +\left( x-u\right) \left( vf\left( x\right) -uf\left(
u\right) \right) }{\left( v-u\right) ^{2}}-\frac{1}{v-u}\int_{u}^{v}f\left(
\mu \right) d\mu \right\vert  \notag \\
&\leq &\frac{\left( v-x\right) ^{2}}{\left( v-u\right) ^{2}}%
\int_{0}^{1}\left( tu+\left( 1-t\right) v\right) \left\vert f^{\prime
}\left( tx+\left( 1-t\right) v\right) \right\vert dt  \notag \\
&&+\frac{\left( x-u\right) ^{2}}{\left( v-u\right) ^{2}}\int_{0}^{1}\left(
tv+\left( 1-t\right) u\right) \left\vert f^{\prime }\left( tx+\left(
1-t\right) u\right) \right\vert dt  \notag \\
&\leq &\frac{\left( v-x\right) ^{2}}{\left( v-u\right) ^{2}}\left(
\int_{0}^{1}\left( tu+\left( 1-t\right) v\right) ^{p}dt\right) ^{\frac{1}{p}%
}\left( \int_{0}^{1}\left\vert f^{\prime }\left( tx+\left( 1-t\right)
v\right) \right\vert ^{q}dt\right) ^{\frac{1}{q}}  \notag \\
&&+\frac{\left( x-u\right) ^{2}}{\left( v-u\right) ^{2}}\left(
\int_{0}^{1}\left( tv+\left( 1-t\right) u\right) ^{p}dt\right) ^{\frac{1}{p}%
}\left( \int_{0}^{1}\left\vert f^{\prime }\left( tx+\left( 1-t\right)
u\right) \right\vert ^{q}dt\right) ^{\frac{1}{q}}.  \notag
\end{eqnarray}%
Since $\left\vert f^{\prime }\right\vert ^{q}$\ is convex on $\left[ u,v%
\right] ,$ by Hadamard's inequality, we have%
\begin{eqnarray}
\int_{0}^{1}\left\vert f^{\prime }\left( tx+\left( 1-t\right) v\right)
\right\vert ^{q}dt &\leq &\frac{\left\vert f^{\prime }\left( x\right)
\right\vert ^{q}+\left\vert f^{\prime }\left( v\right) \right\vert ^{q}}{2}
\label{2a} \\
\int_{0}^{1}\left\vert f^{\prime }\left( tx+\left( 1-t\right) u\right)
\right\vert ^{q}dt &\leq &\frac{\left\vert f^{\prime }\left( x\right)
\right\vert ^{q}+\left\vert f^{\prime }\left( u\right) \right\vert ^{q}}{2}
\label{2b}
\end{eqnarray}%
It can be easily seen that%
\begin{equation}
\int_{0}^{1}\left( tu+\left( 1-t\right) v\right) ^{p}dt=\int_{0}^{1}\left(
tv+\left( 1-t\right) u\right) ^{p}dt=\frac{v^{p+1}-u^{p+1}}{\left(
v-u\right) \left( p+1\right) }=L_{p}^{p}\left( u,v\right)  \label{2c}
\end{equation}%
If expressions (\ref{2a})-(\ref{2c}) are written in (\ref{2}), we obtain%
\begin{eqnarray*}
&&\left\vert \frac{\left( v-x\right) \left( vf\left( v\right) -uf\left(
x\right) \right) +\left( x-u\right) \left( vf\left( x\right) -uf\left(
u\right) \right) }{\left( v-u\right) ^{2}}-\frac{1}{v-u}\int_{u}^{v}f\left(
\mu \right) d\mu \right\vert \\
&\leq &\frac{\left( v-x\right) ^{2}}{\left( v-u\right) ^{2}}L_{p}\left(
u,v\right) \left( \frac{\left\vert f^{\prime }\left( x\right) \right\vert
^{q}+\left\vert f^{\prime }\left( v\right) \right\vert ^{q}}{2}\right) ^{%
\frac{1}{q}}+\frac{\left( x-u\right) ^{2}}{\left( v-u\right) ^{2}}%
L_{p}\left( u,v\right) \left( \frac{\left\vert f^{\prime }\left( x\right)
\right\vert ^{q}+\left\vert f^{\prime }\left( u\right) \right\vert ^{q}}{2}%
\right) ^{\frac{1}{q}}.
\end{eqnarray*}%
The proof is completed.
\end{proof}

\begin{theorem}
\label{t4}Let $f:J\rightarrow 
\mathbb{R}
$ be a differentiable function on $J^{\circ }$. If $\left\vert f^{\prime
}\right\vert ^{q}$\ is $s$-convex on $\left[ u,v\right] $ for some fixed $%
s\in \left( 0,1\right] $ and$\ q>1$ with $1/p+1/q=1$, then%
\begin{eqnarray*}
&&\left\vert \frac{\left( v-x\right) \left( vf\left( v\right) -uf\left(
x\right) \right) +\left( x-u\right) \left( vf\left( x\right) -uf\left(
u\right) \right) }{\left( v-u\right) ^{2}}-\frac{1}{v-u}\int_{u}^{v}f\left(
\mu \right) d\mu \right\vert \\
&\leq &\frac{\left( v-x\right) ^{2}}{\left( v-u\right) ^{2}}L_{p}\left(
u,v\right) \left( \frac{\left\vert f^{\prime }\left( x\right) \right\vert
^{q}+\left\vert f^{\prime }\left( v\right) \right\vert ^{q}}{s+1}\right) ^{%
\frac{1}{q}}+\frac{\left( x-u\right) ^{2}}{\left( v-u\right) ^{2}}%
L_{p}\left( u,v\right) \left( \frac{\left\vert f^{\prime }\left( x\right)
\right\vert ^{q}+\left\vert f^{\prime }\left( u\right) \right\vert ^{q}}{s+1}%
\right) ^{\frac{1}{q}}.
\end{eqnarray*}%
for each $x\in \left[ u,v\right] .$
\end{theorem}

\begin{proof}
The proof of Theorem \ref{t4} is similar to Theorem \ref{t3}.
\end{proof}

\begin{corollary}
\label{c3}In Theorem \ref{t4},

i) if we take $x=u$ or $x=v,$ we get%
\begin{equation}
\left\vert \frac{vf\left( v\right) -uf\left( u\right) }{v-u}-\frac{1}{v-u}%
\int_{u}^{v}f\left( \mu \right) d\mu \right\vert \leq L_{p}\left( u,v\right)
\left( \frac{\left\vert f^{\prime }\left( u\right) \right\vert
^{q}+\left\vert f^{\prime }\left( v\right) \right\vert ^{q}}{s+1}\right) ^{%
\frac{1}{q}}.  \label{c31}
\end{equation}%
ii) if we take $x=\frac{u+v}{2}$ and since $\left( \frac{1}{s+1}\right) ^{%
\frac{1}{q}}\leq 1,$ $s\in \left( 0,1\right] ,$ we get%
\begin{eqnarray}
&&\left\vert \frac{vf\left( v\right) -uf\left( u\right) }{2\left( v-u\right) 
}+\frac{1}{2}f\left( \frac{u+v}{2}\right) -\frac{1}{v-u}\int_{u}^{v}f\left(
\mu \right) d\mu \right\vert  \notag \\
&\leq &\frac{L_{p}\left( u,v\right) }{4}\left[ \left( \frac{\left\vert
f^{\prime }\left( \frac{u+v}{2}\right) \right\vert ^{q}+\left\vert f^{\prime
}\left( v\right) \right\vert ^{q}}{s+1}\right) ^{\frac{1}{q}}+\left( \frac{%
\left\vert f^{\prime }\left( \frac{u+v}{2}\right) \right\vert
^{q}+\left\vert f^{\prime }\left( u\right) \right\vert ^{q}}{s+1}\right) ^{%
\frac{1}{q}}\right]  \label{c32} \\
&\leq &\frac{L_{p}\left( u,v\right) }{4}\left[ \left( \left\vert f^{\prime
}\left( \frac{u+v}{2}\right) \right\vert ^{q}+\left\vert f^{\prime }\left(
v\right) \right\vert ^{q}\right) ^{\frac{1}{q}}+\left( \left\vert f^{\prime
}\left( \frac{u+v}{2}\right) \right\vert ^{q}+\left\vert f^{\prime }\left(
u\right) \right\vert ^{q}\right) ^{\frac{1}{q}}\right]  \label{c33}
\end{eqnarray}
\end{corollary}

\begin{proposition}
Let $u,v\in J^{\circ },\ 0<u<v$ and $s\in \left( 0,1\right] ,$ then%
\begin{eqnarray*}
&&\left\vert \left( s-1\right) L_{s}^{s}\left( u,v\right) +A^{s}\left(
u,v\right) \right\vert \\
&\leq &\frac{L_{p}\left( u,v\right) }{2}\left( \frac{s}{s+1}\right) ^{\frac{1%
}{q}}\left[ \left( \left( \frac{u+v}{2}\right) ^{\left( s-1\right)
q}+v^{\left( s-1\right) q}\right) ^{\frac{1}{q}}+\left( \left( \frac{u+v}{2}%
\right) ^{\left( s-1\right) q}+u^{\left( s-1\right) q}\right) ^{\frac{1}{q}}%
\right]
\end{eqnarray*}
\end{proposition}

\begin{proof}
The proof follows from (\ref{c32}) applied\ to the $s$-convex function $%
f\left( x\right) =x^{s}.$ Equality in (\ref{q}) holds if and only if $s=1$
and $u=v.$
\end{proof}

\begin{theorem}
\label{t5}Let $f:J\rightarrow 
\mathbb{R}
$ be a differentiable function on $J^{\circ }.$ If $\left\vert f^{\prime
}\right\vert ^{q}$\ is $s$-convex on $\left[ u,v\right] $ for some fixed $%
s\in \left( 0,1\right] $\ and $q\geq 1$, then%
\begin{eqnarray*}
&&\left\vert \frac{\left( v-x\right) \left( vf\left( v\right) -uf\left(
x\right) \right) +\left( x-u\right) \left( vf\left( x\right) -uf\left(
u\right) \right) }{\left( v-u\right) ^{2}}-\frac{1}{v-u}\int_{u}^{v}f\left(
\mu \right) d\mu \right\vert \\
&\leq &\frac{\left( v-x\right) ^{2}}{\left( v-u\right) ^{2}}A^{1-\frac{1}{q}%
}\left( u,v\right) \left( \frac{1}{\left( s+1\right) \left( s+2\right) }%
\right) ^{\frac{1}{q}}\left( \left( \left( s+1\right) u+v\right) \left\vert
f^{\prime }\left( x\right) \right\vert ^{q}+\left( \left( s+1\right)
v+u\right) \left\vert f^{\prime }\left( v\right) \right\vert ^{q}dt\right) ^{%
\frac{1}{q}} \\
&&+\frac{\left( x-u\right) ^{2}}{\left( v-u\right) ^{2}}A^{1-\frac{1}{q}%
}\left( u,v\right) \left( \frac{1}{\left( s+1\right) \left( s+2\right) }%
\right) ^{\frac{1}{q}}\left( \left( \left( s+1\right) v+u\right) \left\vert
f^{\prime }\left( x\right) \right\vert ^{q}+\left( \left( s+1\right)
u+v\right) \left\vert f^{\prime }\left( u\right) \right\vert ^{q}dt\right) ^{%
\frac{1}{q}}
\end{eqnarray*}%
for each $x\in \left[ u,v\right] .$
\end{theorem}

\begin{proof}
Using Lemma \ref{ll} and the well-known power mean inequality and since $%
\left\vert f^{\prime }\right\vert ^{q}$\ is $s$-convex on $\left[ u,v\right]
,$ we establish%
\begin{eqnarray*}
&& \\
&&\left\vert \frac{\left( v-x\right) \left( vf\left( v\right) -uf\left(
x\right) \right) +\left( x-u\right) \left( vf\left( x\right) -uf\left(
u\right) \right) }{\left( v-u\right) ^{2}}-\frac{1}{v-u}\int_{u}^{v}f\left(
\mu \right) d\mu \right\vert \\
&\leq &\frac{\left( v-x\right) ^{2}}{\left( v-u\right) ^{2}}%
\int_{0}^{1}\left( tu+\left( 1-t\right) v\right) \left\vert f^{\prime
}\left( tx+\left( 1-t\right) v\right) \right\vert dt+\frac{\left( x-u\right)
^{2}}{\left( v-u\right) ^{2}}\int_{0}^{1}\left( tv+\left( 1-t\right)
u\right) \left\vert f^{\prime }\left( tx+\left( 1-t\right) u\right)
\right\vert dt \\
&\leq &\frac{\left( v-x\right) ^{2}}{\left( v-u\right) ^{2}}\left(
\int_{0}^{1}\left( tu+\left( 1-t\right) v\right) dt\right) ^{1-\frac{1}{q}%
}\left( \int_{0}^{1}\left( tu+\left( 1-t\right) v\right) \left\vert
f^{\prime }\left( tx+\left( 1-t\right) v\right) \right\vert ^{q}dt\right) ^{%
\frac{1}{q}} \\
&&+\frac{\left( x-u\right) ^{2}}{\left( v-u\right) ^{2}}\left(
\int_{0}^{1}\left( tv+\left( 1-t\right) u\right) dt\right) ^{1-\frac{1}{q}%
}\left( \int_{0}^{1}\left( tv+\left( 1-t\right) u\right) \left\vert
f^{\prime }\left( tx+\left( 1-t\right) u\right) \right\vert ^{q}dt\right) ^{%
\frac{1}{q}} \\
&\leq &\frac{\left( v-x\right) ^{2}}{\left( v-u\right) ^{2}}\left(
\int_{0}^{1}\left( tu+\left( 1-t\right) v\right) dt\right) ^{1-\frac{1}{q}%
}\left( \int_{0}^{1}\left( tu+\left( 1-t\right) v\right) \left(
t^{s}\left\vert f^{\prime }\left( x\right) \right\vert ^{q}+\left(
1-t\right) ^{s}\left\vert f^{\prime }\left( v\right) \right\vert ^{q}\right)
dt\right) ^{\frac{1}{q}} \\
&&+\frac{\left( x-u\right) ^{2}}{\left( v-u\right) ^{2}}\left(
\int_{0}^{1}\left( tv+\left( 1-t\right) u\right) dt\right) ^{1-\frac{1}{q}%
}\left( \int_{0}^{1}\left( tv+\left( 1-t\right) u\right) \left(
t^{s}\left\vert f^{\prime }\left( x\right) \right\vert ^{q}+\left(
1-t\right) ^{s}\left\vert f^{\prime }\left( u\right) \right\vert ^{q}\right)
dt\right) ^{\frac{1}{q}} \\
&=&\frac{\left( v-x\right) ^{2}}{\left( v-u\right) ^{2}}\left( \frac{u+v}{2}%
\right) ^{1-\frac{1}{q}}\left( \frac{\left( s+1\right) u+v}{\left(
s+1\right) \left( s+2\right) }\left\vert f^{\prime }\left( x\right)
\right\vert ^{q}+\frac{\left( s+1\right) v+u}{\left( s+1\right) \left(
s+2\right) }\left\vert f^{\prime }\left( v\right) \right\vert ^{q}dt\right)
^{\frac{1}{q}} \\
&&+\frac{\left( x-u\right) ^{2}}{\left( v-u\right) ^{2}}\left( \frac{u+v}{2}%
\right) ^{1-\frac{1}{q}}\left( \frac{\left( s+1\right) v+u}{\left(
s+1\right) \left( s+2\right) }\left\vert f^{\prime }\left( x\right)
\right\vert ^{q}+\frac{\left( s+1\right) u+v}{\left( s+1\right) \left(
s+2\right) }\left\vert f^{\prime }\left( u\right) \right\vert ^{q}dt\right)
^{\frac{1}{q}}
\end{eqnarray*}%
The proof is completed.
\end{proof}

\begin{theorem}
\label{t6}Let $f:J\rightarrow 
\mathbb{R}
$ be a differentiable function on $J^{\circ }.$ If $\left\vert f^{\prime
}\right\vert ^{q}$\ is convex on $\left[ u,v\right] $ and $q\geq 1$, then%
\begin{eqnarray*}
&&\left\vert \frac{\left( v-x\right) \left( vf\left( v\right) -uf\left(
x\right) \right) +\left( x-u\right) \left( vf\left( x\right) -uf\left(
u\right) \right) }{\left( v-u\right) ^{2}}-\frac{1}{v-u}\int_{u}^{v}f\left(
\mu \right) d\mu \right\vert \\
&\leq &\frac{\left( v-x\right) ^{2}}{\left( v-u\right) ^{2}}A^{1-\frac{1}{q}%
}\left( u,v\right) \left( \frac{1}{6}\right) ^{\frac{1}{q}}\left( \left(
2u+v\right) \left\vert f^{\prime }\left( x\right) \right\vert ^{q}+\left(
2v+u\right) \left\vert f^{\prime }\left( v\right) \right\vert ^{q}dt\right)
^{\frac{1}{q}} \\
&&+\frac{\left( x-u\right) ^{2}}{\left( v-u\right) ^{2}}A^{1-\frac{1}{q}%
}\left( u,v\right) \left( \frac{1}{6}\right) ^{\frac{1}{q}}\left( \left(
2v+u\right) \left\vert f^{\prime }\left( x\right) \right\vert ^{q}+\left(
2u+v\right) \left\vert f^{\prime }\left( u\right) \right\vert ^{q}dt\right)
^{\frac{1}{q}}
\end{eqnarray*}%
for each $x\in \left[ u,v\right] .$
\end{theorem}

\begin{proof}
In Theorem \ref{t5}, if we take $s=1,$ then the assertion is proved.
\end{proof}

\begin{corollary}
\label{c4}In Theorem \ref{t5},

i) if we take $x=u$ or $x=v,$ we get%
\begin{eqnarray*}
&&\left\vert \frac{\left( v-u\right) \left( vf\left( v\right) -uf\left(
u\right) \right) }{\left( v-u\right) ^{2}}-\frac{1}{v-u}\int_{u}^{v}f\left(
\mu \right) d\mu \right\vert \\
&\leq &A^{1-\frac{1}{q}}\left( u,v\right) \left( \frac{1}{\left( s+1\right)
\left( s+2\right) }\right) ^{\frac{1}{q}}\left( \left( \left( s+1\right)
u+v\right) \left\vert f^{\prime }\left( u\right) \right\vert ^{q}+\left(
\left( s+1\right) v+u\right) \left\vert f^{\prime }\left( v\right)
\right\vert ^{q}dt\right) ^{\frac{1}{q}}
\end{eqnarray*}%
ii) if we choose $x=u$ or $x=v$ and $s=q=1,$ we get%
\begin{equation*}
\left\vert \frac{\left( v-u\right) \left( vf\left( v\right) -uf\left(
u\right) \right) }{\left( v-u\right) ^{2}}-\frac{1}{v-u}\int_{u}^{v}f\left(
\mu \right) d\mu \right\vert \leq \frac{1}{6}\left( \left( 2u+v\right)
\left\vert f^{\prime }\left( u\right) \right\vert +\left( 2v+u\right)
\left\vert f^{\prime }\left( v\right) \right\vert dt\right)
\end{equation*}%
iii) if we take $x=\frac{u+v}{2}$ and since $\left( \frac{1}{\left(
s+1\right) \left( s+2\right) }\right) ^{\frac{1}{q}}\leq 1,$ $s\in \left( 0,1%
\right] ,$ we get%
\begin{eqnarray}
&&\left\vert \frac{vf\left( v\right) -uf\left( u\right) }{2\left( v-u\right) 
}+\frac{1}{2}f\left( \frac{u+v}{2}\right) -\frac{1}{v-u}\int_{u}^{v}f\left(
\mu \right) d\mu \right\vert  \notag \\
&&  \label{c41} \\
&\leq &\frac{1}{4}A^{1-\frac{1}{q}}\left( u,v\right) \left( \frac{1}{\left(
s+1\right) \left( s+2\right) }\right) ^{\frac{1}{q}}\left( \left( \left(
s+1\right) u+v\right) \left\vert f^{\prime }\left( \frac{u+v}{2}\right)
\right\vert ^{q}+\left( \left( s+1\right) v+u\right) \left\vert f^{\prime
}\left( v\right) \right\vert ^{q}dt\right) ^{\frac{1}{q}}  \notag \\
&&+\frac{1}{4}A^{1-\frac{1}{q}}\left( u,v\right) \left( \frac{1}{\left(
s+1\right) \left( s+2\right) }\right) ^{\frac{1}{q}}\left( \left( \left(
s+1\right) v+u\right) \left\vert f^{\prime }\left( \frac{u+v}{2}\right)
\right\vert ^{q}+\left( \left( s+1\right) u+v\right) \left\vert f^{\prime
}\left( u\right) \right\vert ^{q}dt\right) ^{\frac{1}{q}}  \notag \\
&&  \label{c42} \\
&\leq &\frac{1}{4}A^{1-\frac{1}{q}}\left( u,v\right) \left( \left( \left(
s+1\right) u+v\right) \left\vert f^{\prime }\left( \frac{u+v}{2}\right)
\right\vert ^{q}+\left( \left( s+1\right) v+u\right) \left\vert f^{\prime
}\left( v\right) \right\vert ^{q}dt\right) ^{\frac{1}{q}}  \notag \\
&&+\frac{1}{4}A^{1-\frac{1}{q}}\left( u,v\right) \left( \left( \left(
s+1\right) v+u\right) \left\vert f^{\prime }\left( \frac{u+v}{2}\right)
\right\vert ^{q}+\left( \left( s+1\right) u+v\right) \left\vert f^{\prime
}\left( u\right) \right\vert ^{q}dt\right) ^{\frac{1}{q}}  \notag
\end{eqnarray}
\end{corollary}

\begin{proposition}
Let $u,v\in J^{\circ },\ 0<u<v$ and $s\in \left( 0,1\right] ,$ then%
\begin{eqnarray*}
&&\left\vert \left( s-1\right) L_{s}^{s}\left( u,v\right) +A^{s}\left(
u,v\right) \right\vert \\
&\leq &\left( \frac{s^{q}A^{q-1}\left( u,v\right) }{\left( s+1\right) \left(
s+2\right) 2^{q}}\right) ^{\frac{1}{q}}\left( \left( \left( s+1\right)
u+v\right) A^{\left( s-1\right) q}\left( u,v\right) +\left( \left(
s+1\right) v+u\right) v^{\left( s-1\right) q}dt\right) ^{\frac{1}{q}} \\
&&+\left( \frac{s^{q}A^{q-1}\left( u,v\right) }{\left( s+1\right) \left(
s+2\right) 2^{q}}\right) ^{\frac{1}{q}}\left( \left( \left( s+1\right)
v+u\right) A^{\left( s-1\right) q}\left( u,v\right) +\left( \left(
s+1\right) u+v\right) u^{\left( s-1\right) q}dt\right) ^{\frac{1}{q}}
\end{eqnarray*}
\end{proposition}

\begin{proof}
The proof follows from (\ref{c41}) applied\ to the $s$-convex function $%
f\left( x\right) =x^{s}.$
\end{proof}

\begin{theorem}
\label{t7}Let $f:J\rightarrow 
\mathbb{R}
$ be a differentiable function on $J^{\circ }.$ If $\left\vert f^{\prime
}\right\vert ^{q}$\ is $s$-concave on $\left[ u,v\right] $ for some fixed $%
s\in \left( 0,1\right] $\ and $q>1$ with $1/p+1/q=1$, then%
\begin{eqnarray*}
&&\left\vert \frac{\left( v-x\right) \left( vf\left( v\right) -uf\left(
x\right) \right) +\left( x-u\right) \left( vf\left( x\right) -uf\left(
u\right) \right) }{\left( v-u\right) ^{2}}-\frac{1}{v-u}\int_{u}^{v}f\left(
\mu \right) d\mu \right\vert \\
&\leq &\frac{2^{\frac{s-1}{q}}L_{p}\left( u,v\right) }{\left( v-u\right) ^{2}%
}\left[ \left( v-x\right) ^{2}\left\vert f^{\prime }\left( \frac{x+v}{2}%
\right) \right\vert +\left( x-u\right) ^{2}\left\vert f^{\prime }\left( 
\frac{x+u}{2}\right) \right\vert \right]
\end{eqnarray*}%
for each $x\in \left[ u,v\right] .$
\end{theorem}

\begin{proof}
Using Lemma \ref{ll} and using the H\"{o}lder inequality and since $%
\left\vert f^{\prime }\right\vert ^{q}$\ is $s$-concave on $\left[ u,v\right]
$ and using the inequality (\ref{109}), we establish%
\begin{eqnarray*}
&& \\
&&\left\vert \frac{\left( v-x\right) \left( vf\left( v\right) -uf\left(
x\right) \right) +\left( x-u\right) \left( vf\left( x\right) -uf\left(
u\right) \right) }{\left( v-u\right) ^{2}}-\frac{1}{v-u}\int_{u}^{v}f\left(
\mu \right) d\mu \right\vert \\
&\leq &\frac{\left( v-x\right) ^{2}}{\left( v-u\right) ^{2}}%
\int_{0}^{1}\left( tu+\left( 1-t\right) v\right) \left\vert f^{\prime
}\left( tx+\left( 1-t\right) v\right) \right\vert dt+\frac{\left( x-u\right)
^{2}}{\left( v-u\right) ^{2}}\int_{0}^{1}\left( tv+\left( 1-t\right)
u\right) \left\vert f^{\prime }\left( tx+\left( 1-t\right) u\right)
\right\vert dt \\
&\leq &\frac{\left( v-x\right) ^{2}}{\left( v-u\right) ^{2}}\left(
\int_{0}^{1}\left( tu+\left( 1-t\right) v\right) ^{p}dt\right) ^{\frac{1}{p}%
}\left( \int_{0}^{1}\left\vert f^{\prime }\left( tx+\left( 1-t\right)
v\right) \right\vert ^{q}dt\right) ^{\frac{1}{q}} \\
&&+\frac{\left( x-u\right) ^{2}}{\left( v-u\right) ^{2}}\left(
\int_{0}^{1}\left( tv+\left( 1-t\right) u\right) ^{p}dt\right) ^{\frac{1}{p}%
}\left( \int_{0}^{1}\left\vert f^{\prime }\left( tx+\left( 1-t\right)
u\right) \right\vert ^{q}dt\right) ^{\frac{1}{q}} \\
&\leq &\frac{\left( v-x\right) ^{2}}{\left( v-u\right) ^{2}}L_{p}\left(
u,v\right) \left( 2^{s-1}\left\vert f^{\prime }\left( \frac{x+v}{2}\right)
\right\vert ^{q}\right) ^{\frac{1}{q}}+\frac{\left( x-u\right) ^{2}}{\left(
v-u\right) ^{2}}L_{p}\left( u,v\right) \left( 2^{s-1}\left\vert f^{\prime
}\left( \frac{x+u}{2}\right) \right\vert ^{q}\right) ^{\frac{1}{q}}
\end{eqnarray*}%
The proof is completed.
\end{proof}

\begin{theorem}
\label{t8}Let $f:J\rightarrow 
\mathbb{R}
$ be a differentiable function on $J^{\circ }.$ If $\left\vert f^{\prime
}\right\vert ^{q}$\ is concave on $\left[ u,v\right] $ and $q>1$ with $%
1/p+1/q=1$, then%
\begin{eqnarray*}
&&\left\vert \frac{\left( v-x\right) \left( vf\left( v\right) -uf\left(
x\right) \right) +\left( x-u\right) \left( vf\left( x\right) -uf\left(
u\right) \right) }{\left( v-u\right) ^{2}}-\frac{1}{v-u}\int_{u}^{v}f\left(
\mu \right) d\mu \right\vert \\
&\leq &\frac{L_{p}\left( u,v\right) }{\left( v-u\right) ^{2}}\left[ \left(
v-x\right) ^{2}\left\vert f^{\prime }\left( \frac{x+v}{2}\right) \right\vert
+\left( x-u\right) ^{2}\left\vert f^{\prime }\left( \frac{x+u}{2}\right)
\right\vert \right]
\end{eqnarray*}%
for each $x\in \left[ u,v\right] .$
\end{theorem}

\begin{proof}
Using Lemma \ref{ll} and using the H\"{o}lder inequality and since $%
\left\vert f^{\prime }\right\vert ^{q}$\ is concave on $\left[ u,v\right] $
and using Hadamard's inequality for concave functions, we complete the
proof. Or, in Theorem \ref{t7}, if we take $s=1,$ then the assertion is also
proved.
\end{proof}

\begin{corollary}
\label{c5}In Theorem \ref{t7},

i) if we take $x=u$ or $x=v,$ we get%
\begin{equation*}
\left\vert \frac{\left( v-u\right) \left( vf\left( v\right) -uf\left(
u\right) \right) }{\left( v-u\right) ^{2}}-\frac{1}{v-u}\int_{u}^{v}f\left(
\mu \right) d\mu \right\vert \leq 2^{\frac{s-1}{q}}L_{p}\left( u,v\right)
\left\vert f^{\prime }\left( \frac{u+v}{2}\right) \right\vert
\end{equation*}%
ii) if we take $x=\frac{u+v}{2},$ we get%
\begin{eqnarray}
&&\left\vert \frac{vf\left( v\right) -uf\left( u\right) }{2\left( v-u\right) 
}+\frac{1}{2}f\left( \frac{u+v}{2}\right) -\frac{1}{v-u}\int_{u}^{v}f\left(
\mu \right) d\mu \right\vert  \label{c51} \\
&\leq &\frac{2^{\frac{s-1}{q}}L_{p}\left( u,v\right) }{4}\left[ \left\vert
f^{\prime }\left( \frac{u+3v}{4}\right) \right\vert +\left\vert f^{\prime
}\left( \frac{v+3u}{4}\right) \right\vert \right]  \notag
\end{eqnarray}
\end{corollary}

\begin{proposition}
Let $u,v\in J^{\circ },\ 0<u<v$ and $s=1,$ then%
\begin{eqnarray*}
&&\left\vert \frac{v\sin v-u\sin u+2\cos v-2\cos u}{v-u}+\sin \left( A\left(
u,v\right) \right) \right\vert \\
&\leq &\frac{L_{p}\left( u,v\right) }{2}\left[ \left\vert \cos \left( \frac{%
u+3v}{4}\right) \right\vert +\left\vert \cos \left( \frac{v+3u}{4}\right)
\right\vert \right]
\end{eqnarray*}
\end{proposition}

\begin{proof}
The proof follows from (\ref{c51}) applied\ to the concave function $f:\left[
0,\pi \right] \rightarrow \left[ 0,1\right] ,$ $f\left( x\right) =\sin x.$
\end{proof}

\end{document}